\documentclass{aims}
\usepackage{amsmath}
  \usepackage{paralist}
  \usepackage{graphics} %% add this and next lines if pictures should be in esp format
  \usepackage{epsfig} %For pictures: screened artwork should be set up with an 85 or 100 line screen
\usepackage{graphicx}  \usepackage{epstopdf}%This is to transfer .eps figure to .pdf figure; please compile your paper using PDFLeTex or PDFTeXify.
 \usepackage[colorlinks=true]{hyperref}
\hypersetup{urlcolor=blue, citecolor=red}

\def\currentvolume{X}
 \def\currentissue{X}
  \def\currentyear{200X}
   \def\currentmonth{XX}
    \def\ppages{X--XX}
     \def\DOI{10.3934/xx.xx.xx.xx}

\newtheorem{theorem}{Theorem}[section]

\newtheorem{lemma}[theorem]{Lemma}

\theoremstyle{definition}
\newtheorem{definition}[theorem]{Definition}
\newtheorem{remark}{Remark}

\newtheorem{algorithm}{Algorithm}

\title[An extension of hybrid method without extrapolation step to EPs] %Use the shortened version of the full title
      {An extension of hybrid method without extrapolation step to equilibrium problems}

\author[Van Hieu Dang]{}


 \email{dv.hieu83@gmail.com}
\thanks{$^*$ Corresponding author: dv.hieu83@gmail.com}

\begin{document}
\maketitle

% Enter the first author's name and address:
\centerline{\scshape Van Hieu Dang$^*$}
\medskip
{\footnotesize
% please put the address of the first author
 \centerline{Department of Mathematics, Vietnam National University}
   \centerline{334 - Nguyen Trai Street, Thanh Xuan, Ha Noi, Viet Nam}
  % \centerline{ Springfield, MO 65801-2604, USA}
} % Do not forget to end the {\footnotesize by the sign }

\medskip
% 
% \centerline{\scshape First-name2 last-name2 and First-name3
% last-name3}
% \medskip
% {\footnotesize
%  % please put the address of the second  and third author
%  \centerline{ First line of the address of the second author}
%    \centerline{Other lines}
%    \centerline{Springfield, MO 65810, USA}
% }
% 
\bigskip

% The name of the associate editor will be entered by an editorial staff
% "Communicated by the associate editor name" is not needed for special issue.
 \centerline{(Communicated by the associate editor name)}

%The abstract of your paper
\begin{abstract}
In this paper, we introduce a new hybrid algorithm for solving equilibrium problems. The algorithm combines the extragradient method and the 
hybrid (outer approximation) method.
In this algorithm, only an optimization program is solved at each iteration without the extra-steps like as in the extragradient method and the 
Armijo linesearch method. A specially constructed half-space in the hybrid method is the reason for the absence of an 
optimization program in our algorithm. The strong convergence theorem is established and several numerical experiments are implemented to illustrate 
the convergence of the algorithm and compare it with others.
\end{abstract}

\section{Introduction}\label{intro}
The equilibrium problem (EP) \cite{BO1994} which was considered as the Ky Fan inequality \cite{KF1972} is very general in the sense that it includes, as special cases, many mathematical 
models such as: variational inequalities, fixed point problems, optimization problems, Nash equilirium point problems, complementarity problems, see 
\cite{BO1994,MO1992} and the references therein. Many methods have been proposed for solving EPs 
\cite{AH2014b,BO1994,CH2005,DHM2014,H2015,M2000,MO1992,QMH2008,TT2007,JK2010}. 
The most solution approximations to EPs are often based on the resolvent of equilibrium bifunction (see, for instance \cite{CH2005}) in which 
a strongly monotone regularization equilibrium problem (REP) is solved at each iterative step. It is also called the proximal point method (PPM). This method was first introduced by 
Martinet \cite{M1970} for variational inequalities, and then it was extended by Rockafellar \cite{R1976} for finding a zero point of a monotone operator. In 2000, 
Konnov \cite{K2000} further extended PPM to Ky Fan inequalities for monotone or weakly monotone bifunctions.

A special case of EP is the variational inequality problem (VIP). The projection plays an important role in constrained optimization 
problems. The simplest method for VIPs is the gradient projection method in which only a projection on the feasible set is computed. However, in order to 
obtain the convergence, the method requires the restrictive assumption that operators are strongly (or inverse strongly) monotone. To overcome this, 
Korpelevich \cite{K1976} introduced the extragradient method (double projection method) where two metric projections onto the feasible set be implemented 
at each iteration. The convergence of the extragradient method was proved under the weaker assumption that operators are only monotone 
(even, pseudomonotone) and $L$ - Lipschitz continuous. Some extragradient-like algorithms proposed for solving VIPs can 
be found in \cite{CHW2010,CY2007,HYY2004,NT2006a,NT2006b} and the references therein.  However, the projection is only found easily if the 
constrained set has a simple structure, for instance, as balls, hyperplanes or halfspaces. So, several modifications of the extragradient method have been 
proposed in various ways \cite{CGR2011a,CGR2011,GR1984,MS2015}. For instance, the authors 
in \cite{CGR2011a} replaced the second projection onto the feasible set in the extragradient method by 
one onto a half-space and proposed the subgradient extragradient method for VIPs in Hilbert spaces.

In recent years, Korpelevich's extragradient method has been naturally
extended to EPs for monotone (more general, pseudomonotone) and Lipschitz-type continuous bifunctions and widely studied both theoretically and 
algorithmically \cite{DHM2014,H2015a,VSH2013,NSNN2015,QMH2008,SVH2011,VSN2012}. In the extended extragradient methods 
to EPs, we need to solve two strongly convex optimization programs on a closed convex constrained set (see, Algorithms 
\ref{VSH2013a}, \ref{VSH2013b} and \ref{DHM2014} in Section \ref{algor}). They are generalizations of two projections in Korpelevich's 
extragradient method. The advantage of the extragradient method is that two optimization programs are solved at each iteration which seems to be 
numerically easier than the non-linear inequality (or REP) in PPM. However, this might still be costly and affects the efficiency of the used method 
if the structure of feasible set and equilibrium bifunction are complex. Moreover, we are not aware of any modification of the extragradient method 
for EPs.

In this paper, motivated by the \textit{hybrid method without the extrapolation step} \cite{MS2015} for variational inequalities, the \textit{extragradient 
method} \cite{QMH2008} and the \textit{hybrid method}, we have proposed a new hybrid algorithm 
for solving EPs. In this algorithm, by constructing a specially cutting - halfspace in the hybrid method, we only need to solve a strongly convex optimization program 
onto the feasible set at each iteration. The absence of an optimization program in our algorithm (compare with the extragradient method) can be considered an 
improvement of the results in \cite{DHM2014,H2015a,LSV2011,VSH2013,NSNN2015,SVH2011,VSN2012}.

The remainder of the paper is organized as follows: Section $\ref{algor}$ introduce our algorithm and some related works.  In Section $\ref{pre}$, 
we collect some definitions and preliminary results used in the paper. Section $\ref{main}$ deals 
with proving the convergence of the algorithm. Some applications of our algorithm to Gato differentiable EPs and multivalued variational 
inequalities are presented in Section \ref{Appl}. 
Finally, in Section $\ref{Numer}$ we provide some numerical examples to illustrate the 
convergence of the proposed algorithm and compare it with others.
%%%%%%%%%%%%%%%%%%%%%%%%%%%%%%%%%%%
\section{Algorithm and related works} \label{algor}
Let $H$ be a real Hilbert space, $C$ be a nonempty closed convex subset of  $H$ and $f:C\times C\to \Re$ be a bifunction with $f(x,x)=0$ for all $x\in C$. 
The equilibrium problem (EP) for the bifunction $f$ on $C$ is to find $x^*\in C$ such that
\begin{equation}\label{eq:EP}
f(x^*,y)\ge 0,~\forall y\in C.
\end{equation}
The solution set of EP $(\ref{eq:EP})$ is denoted by $EP(f,C)$. In this paper, we introduce 
the following hybrid algorithm for solving EP $(\ref{eq:EP})$.
\begin{algorithm}[An extended hybrid algorithm without extrapolation step]\label{H2015}
$$
\left \{
\begin{array}{ll}
y_{n+1}=  \underset{y\in C}{\rm argmin} \{ \lambda f(y_n, y) +\frac{1}{2}||x_n-y||^2\},\\
x_{n+1}=P_{C_n\cap Q_n}(x_0),
\end{array}
\right.
$$
\end{algorithm}
\noindent where $C_n,Q_n$ are two specially constructed half-spaces (see Algorithm $\ref{algor1}$ in Section $\ref{main}$ below). 

In the special case, $f(x,y)=\left\langle A(x),y-x\right\rangle$ where $A:C\to H$ is a nonlinear operator 
then EP becomes the following variational inequality problem (VIP): Find $x^*\in C$ such that
\begin{equation}\label{VIP}
\left\langle A(x^*),y-x^*\right\rangle\ge 0,~\forall y\in C.
\end{equation}
Then, our algorithm (Algorithm $\ref{H2015}$) becomes the following \textit{hybrid algorithm without extrapolation step} which was introduced 
in \cite{MS2015} for VIPs.
\begin{algorithm}[The hybrid algorithm without extrapolation step]\label{MS2015}
$$
\left \{
\begin{array}{ll}
y_{n+1}=  P_C\left(x_n-\lambda A(y_n)\right),\\
x_{n+1}=P_{C_n\cap Q_n}(x_0).
\end{array}
\right.
$$
\end{algorithm}
In 2008, Quoc et al. \cite{QMH2008} extended Korpelevich's extragradient method to EPs in Euclidean spaces in which 
two optimization programs are solved at each iteration. Recently, Nguyen et al. \cite{NSNN2015} also have done in that direction 
and proposed the general extragradient method which consists of solving three optimization programs on the feasible set. In Euclidean spaces, 
the convergence of the sequences generated by the extragradient methods \cite{NSNN2015,QMH2008} was proved under 
the assumptions of pseudomonotonicity and Lipschitz-type continuity of equilibrium bifunctions. The problem which arises in 
infinite dimensional Hilbert spaces is how to design an algorithm which provides the strong convergence.
% \begin{algorithm}\label{NSNN2015}
% $$
% \left \{
% \begin{array}{ll}
% y_{n}=  \underset{y\in C}{\rm argmin} \{ \alpha_n f(x_n, y) +\frac{1}{2}||x_n-y||^2\},\\
% z_{n}=  \underset{y\in C}{\rm argmin} \{ \beta_n f(y_n, y) +\frac{1}{2}||y_n-y||^2\},\\
% x_{n+1}=  \underset{y\in C}{\rm argmin} \{ \beta_n f(z_n, y) +\frac{1}{2}||x_n-y||^2\},
% \end{array}
% \right.
% $$
% \end{algorithm}
% \noindent where $\alpha_n,\beta_n$ are given parameters. 
In 2012, Vuong et al. \cite{VSN2012} used the extragradient method in \cite{QMH2008} and the hybrid (outer approximation) method 
to obtain the following strong convergence hybrid algorithm 
\begin{algorithm}\label{VSH2013a}
$$
\left \{
\begin{array}{ll}
y_n=  \underset{y\in C}{\rm argmin} \{ \lambda f(x_n, y) +\frac{1}{2}||x_n-y||^2\},\\ 
z_n = \underset{y\in C}{\rm argmin}\{ \lambda f(y_n, y) +\frac{1}{2}||x_n-y||^2\},\\
C_{n}=\left\{z\in C:||z_n-z||^2\le||x_n-z||^2\right\},\\
Q_n=\left\{z\in C:\left\langle x_0-x_n,z-x_n\right\rangle\le 0\right\},\\
x_{n+1}=P_{C_n\cap Q_n}(x_0).
\end{array}
\right.
$$
\end{algorithm}
\noindent In 2013, another hybrid algorithm \cite[Algorithm 1]{VSH2013} was also proposed in this direction as
\begin{algorithm}\label{VSH2013b}
$$
\left \{
\begin{array}{ll}
y_n=  \underset{y\in C}{\rm argmin} \{ \lambda f(x_n, y) +\frac{1}{2}||x_n-y||^2\},\\ 
z_n = \underset{y\in C}{\rm argmin}\{ \lambda f(y_n, y) +\frac{1}{2}||x_n-y||^2\},\\
C_{n+1}=\left\{z\in C_n:||z_n-z||^2\le||x_n-z||^2\right\},\\
x_{n+1}=P_{C_{n+1}}(x_0)
\end{array}
\right.
$$
\end{algorithm}
\noindent The authors in \cite{VSH2013,VSN2012} proved that the sequences $\left\{x_n\right\}$ generated by Algorithms $\ref{VSH2013a}$ 
and $\ref{VSH2013b}$ converges strongly to  $P_{EP(f,C)}(x_0)$. Note that, the set $C_{n+1}$ in Algorithm $\ref{VSH2013b}$, in general, 
is not easy to construct. 

In 2014, in order to avoid the condition of the Lipschitz-type continuity the bifunction $f$, Dinh et al. \cite{DHM2014} replaced the second 
optimization problem in the extragradient method by the Armijo linesearch technique and obtained the following hybrid algorithm
\begin{algorithm}\label{DHM2014}
$$
\left \{
\begin{array}{ll}
y_n=  \underset{y\in C}{\rm argmin} \{ \lambda f(x_n, y) +\frac{1}{2}||x_n-y||^2\},\\ 
m_n~ \mbox{\rm is the smallest integer number such that}\\
\begin{cases}
z_n= (1-\eta^{m_n})x_n+\eta^{m_n}y_n,\\ 
f( z_n,y_n )+\frac{1}{2\lambda}||x_n-y_n||^2\le 0,
\end{cases}\\
u_n=P_C(x_n-\sigma_n g_n),\\
C_{n}=\left\{z\in C:||u_n-z||^2\le||x_n-z||^2\right\},\\
Q_n=\left\{z\in C:\left\langle x_0-x_n,z-x_n\right\rangle\le 0\right\},\\
x_{n+1}=P_{C_n\cap Q_n}(x_0),
\end{array}
\right.
$$
\end{algorithm}
\noindent where $g_n\in \partial f_2(z_n,z_n)$ and $\sigma_n=-\eta^{m_n}f(z_n,y_n)/(1-\eta^{m_n})||g_n||^2$. 
Arcording to Algorithm $\ref{DHM2014}$, we still have to solve an optimization program on $C$ for $y_n$, find an optimization direction for $z_n$ 
and compute a projection onto $C$ for $u_n$ at each step. We emphasize that the projection $P_{C_n\cap Q_n}(x_0)$ in Algorithms $\ref{VSH2013a}$, $\ref{VSH2013b}$ and $\ref{DHM2014}$ deals with 
the constrained set $C$, while the sets $C_n$ and $Q_n$ in Algorithm 
$\ref{H2015}$ are two half-spaces, and so $x_{n+1}$ can be expressed by an explicit formula (see, for instance \cite{CH2005,SS2000}).
%%%%%%%%%%%%%%%%%%%%%%%%%%%%%%%%%%
%%%%%%%%%%%%%%%%%%%%%%%%%%%%%%%%%%%%%%%%%%%%%%%%%%
\section{Preliminaries}\label{pre}
%%%%%%%%%%%%%%%%%%%%%%%%%%%%%%%%%%%%%
In this section, we recall some definitions and results for further use. Let $C$ be a nonempty closed convex subset of a real Hilbert space $H$. 
% A mapping $S:C\to H$ is called Lipschitz continuous on $C$ if there exists a positive constant $L$ such that $||S(x)-S(y)||\le L||x-y||$ for all $x,y\in C$. If $L=1$, then $S$ is said to be nonexpansive on $C$. The fixed point set of $S$ is denoted by $F(S)$. We begin with the following properties of nonexpansive mappings.
% \begin{lemma}\cite{GK1990}\label{lem.demiclose}
% Assume that $S:C\to H$ is a nonexpansive mapping. If $S$ has a fixed point , then 
% \begin{enumerate}
% \item [$(i)$] $F(S)$ is closed convex subset of $C$.
% \item [$(ii)$] $I-S$ is demiclosed, i.e., whenever $\left\{x_n\right\}$ is a sequence in $C$ weakly converging to some $x\in C$ and the 
% sequence $\left\{(I-S)x_n\right\}$ strongly converges to some $y$ , it follows that $(I-S)x=y$.
% \end{enumerate}
% \end{lemma}
%%%%%%%%%%%%%%%%%%%%%%%%%%%%%%
We begin with some concepts of the monotonicity of a bifunction (see \cite{BO1994,MO1992} for more details).
\begin{definition} A bifunction $f:C\times C\to \Re$ is said to be
\begin{itemize}
\item [$\rm i.$] strongly monotone on $C$ if there exists a constant $\gamma>0$ such that
$$ f(x,y)+f(y,x)\le -\gamma ||x-y||^2,~\forall x,y\in C; $$
\item [$\rm ii.$] monotone on $C$ if 
$$ f(x,y)+f(y,x)\le 0,~\forall x,y\in C; $$
\item [$\rm iii.$] pseudomonotone on $C$ if 
$$ f(x,y)\ge 0 \Longrightarrow f(y,x)\le 0,~\forall x,y\in C;$$
\item [$\rm iv.$] Lipschitz-type continuous on $C$ if there exist two positive constants $c_1,c_2$ such that
$$ f(x,y) + f(y,z) \geq f(x,z) - c_1||x-y||^2 - c_2||y-z||^2, ~ \forall x,y,z \in C.$$
\end{itemize}
\end{definition}
From the definitions above, it is clear that a strongly monotone bifunction is monotone and a monotone bifunction is pseudomonotone, i.e., 
$i.\Longrightarrow ii. \Longrightarrow iii.$ For solving EP $(\ref{eq:EP})$, we assume that the bifunction $f$ satisfies the following conditions:
\begin{itemize}
\item[(A1).] $f$ is pseudomonotone on $C$ and $f(x,x)=0$ for all $x\in C$;
\item [(A2).]  $f$ is Lipschitz-type continuous on $C$ with the constants $c_1,c_2$;
\item [(A3).]   $f$ is weakly continuous on $C\times C$;
\item [(A4).]  $f(x,.)$ is convex and subdifferentiable on $C$  for every fixed $x\in C.$
\end{itemize}  
It is easy to show that under the 
assumptions $\rm (A1), (A3),(A4)$, the solution set $EP(f,C)$ of EP $(\ref{eq:EP})$  is closed and convex 
(see, for instance \cite{QMH2008}). In this paper, we assume that the solution set $EP(f,C)$ is nonempty.

% \begin{lemma}\label{EPexistence}\cite{BS1996}
% If the bifunction $f$ satisfies Assumptions $(A1)-(A4)$, then the solution set $EP(f)$ is closed and convex.
% \end{lemma}
The metric projection $P_C:H\to C$ is defined by
\begin{equation*}\label{eq:1.3}
P_C x=\arg\min\left\{\left\|y-x\right\|:y\in C\right\}.
\end{equation*}
Since $C$ is nonempty, closed and convex, $P_Cx$ exists and is unique. It is also known that $P_C$ has the following characteristic properties, 
see \cite{GR1984} for more details.
\begin{lemma}\label{lem.PropertyPC}
Let $P_C:H\to C$ be the metric projection from $H$ onto $C$. Then
\begin{itemize}
\item [$(i)$] $P_C$ is firmly nonexpansive, i.e.,
\begin{equation*}\label{eq:FirmlyNonexpOfPC}
\left\langle P_C x-P_C y,x-y \right\rangle \ge \left\|P_C x-P_C y\right\|^2,~\forall x,y\in H.
\end{equation*}
\item [$(ii)$] For all $x\in C, y\in H$,
\begin{equation}\label{eq:ProperOfPC}
\left\|x-P_C y\right\|^2+\left\|P_C y-y\right\|^2\le \left\|x-y\right\|^2.
\end{equation}
\item [$(iii)$] $z=P_C x$ if and only if 
\begin{equation}\label{eq:EquivalentPC}
\left\langle x-z,z-y \right\rangle \ge 0,\quad \forall y\in C.
\end{equation}
\end{itemize}
\end{lemma}
%%%%%%%%%%%%%%%%%%%%%%%%%%%%%%%%%%%%%%%
Let $g:C\to \Re$ be a function. The subdifferential of $g$ at $x$ is defined by
\begin{equation*}\label{eq:Subdifferential}
\partial g(x)=\left\{ w\in H: g(y)-g(x)\ge \left\langle w, y-x\right\rangle,~\forall y\in C\right\}.
\end{equation*}
We recall that the normal cone of $C$ at $x\in C$ is defined by
\begin{equation*}\label{eq:NormalCone}
N_C(x)=\left\{w\in H:\left\langle w,y-x  \right\rangle \le 0,~\forall y\in C\right\}.
\end{equation*}
%%%%%%%%%%%%%%%%%%%%%%%%%%%%%
\begin{definition}[Weakly lower semicontinuity]
A function $\varphi: H\to \Re$ is called weakly lower semicontinuous at $x\in H$ if for any sequence $\left\{x_n\right\}$ in $H$ converges weakly to $x$ then 
$$\varphi(x)\le \lim\inf_{n\to\infty}\varphi(x_n). $$
\end{definition}
It is well-known that the functional $\varphi(x):=||x||^2$ is convex and weakly lower semicontinuous. Any Hilbert space has the Kadec-Klee property (see, 
for instance \cite{GK1990}), i.e., if $\left\{x_n\right\}$ is a sequence in $H$ such that $x_n\rightharpoonup x$ 
and $||x_n||\to ||x||$ then $x_n\to x$ as $n\to\infty$.
%%%%%%%%%%%%%%%%%%%%%%%%%%%%%%%%%%%%%%%%%

Finally, we have the following technical lemma.
\begin{lemma}\cite{MS2015}\label{lem.technique}
Let $\left\{\alpha_n\right\}$, $\left\{\beta_n\right\}$, $\left\{\gamma_n\right\}$ be nonnegative real sequences, 
$\alpha,~\beta \in \Re$ and for all $n\ge 0$ the following inequality holds
$$ \alpha_n\le \beta_n+\beta \gamma_n-\alpha \gamma_{n+1}. $$
If $\sum_{n=0}^\infty \beta_n <+\infty$ and $\alpha>\beta\ge 0$ then $\lim_{n\to\infty}\alpha_n=0$.
\end{lemma}
%%%%%%%%%%%%%%%%%%%%%%%%%%%%%%%%%%%%%%%%%%%
\section{Convergence analysis}\label{main}
\setcounter{theorem}{0}
\setcounter{remark}{0}
\setcounter{corollary}{0}
\setcounter{algorithm}{0}
In this section, we present our algorithm for more details and prove its convergence.
\begin{algorithm}[An extended hybrid algorithm without extrapolation step]\label{algor1}
\noindent \textbf{Initialization.} Chose $x_0=x_1 \in H, ~y_0=y_1\in C$ and set $C_0=Q_0=H$. The parameters $\lambda$ and $k$ satisfy the following 
conditions
$$0< \lambda <\frac{1}{2(c_1+c_2)},~ k>\frac{1}{1-2\lambda(c_1+c_2)}.$$
\textbf{Step 1.} Solve a strongly convex program
$$
y_{n+1}=  \underset{y\in C}{\rm argmin} \{ \lambda f(y_n, y) +\frac{1}{2}||x_n-y||^2\}.
$$
If $y_{n+1}=y_n=x_n$ then stop. \\
\textbf{Step 2.} Compute $x_{n+1}=P_{C_n\cap Q_n}(x_0),$
where 
\begin{eqnarray*}
&&C_n=\left\{z\in H: ||y_{n+1} - z||^2\leq ||x_n-z||^2+\epsilon_n \right\},\\
&&Q_n=\left\{z\in H: \left\langle x_0-x_n,z-x_n\right\rangle\le 0\right\},
\end{eqnarray*}
and $\epsilon_n=k||x_n-x_{n-1}||^2+2\lambda c_2||y_n-y_{n-1}||^2-(1-\frac{1}{k}-2\lambda c_1)||y_{n+1}-y_{n}||^2$. 
Set $n:=n+1$ and go back \textbf{Step 1.}
\end{algorithm}
%%%%%%%%%%%%%%%%%%%%%%%%%%
We have the following result.
%%%%%%%%%%%%%%%%%%%%%%%%%%%%%%%%%
\begin{lemma}\label{lem2}
If Algorithm $\ref{algor1}$ finishes at the iteration step $n<\infty$, then $x_n\in EP(f,C)$.
\end{lemma}
%%%%%%%%%%%%%%%%%%%%%%%%%%%
\begin{proof}
Assume that $y_{n+1}=y_n=x_n$. From the definition of $y_{n+1}$,
$$ x_n=\arg\min\left\{\lambda_n f(x_n,y)+\frac{1}{2}||x_n-y||^2: y\in C\right\}. $$
Thus, from \cite[Proposition 2.1]{M2000}, one has $x_n\in EP(f,C)$ . The proof of Lemma  $\ref{lem2}$ is complete.
\end{proof}
%%%%%%%%%%%%%%%%%%%%%%%%%%%
%%%%%%%%%%%%%%%%%%%%%%%%%%%%%%%%%%%%%
We need the lemma below which is an infinite version of Theorem 27.4 in \cite{R1970} and is similarly proved by using
Moreau-Rockafellar Theorem to find the subdifferential of a sum of a convex function $g$ and the indicator function 
$\delta_C$ to $C$ in a real Hilbert space.
%%%%%%%%%%%%%%%%%%%%%%%%%%%%%%%%%%%%
\begin{lemma}\label{lem.Equivalent_MinPro}
Let $C$ be a convex subset of a real Hilbert space H and $g:C\to \Re$ be a convex and subdifferentiable function on $C$. Then, 
$x^*$ is a solution to the following convex optimization problem
\begin{equation*}\min\left\{g(x):x\in C\right\}
\end{equation*}
if and only if  ~  $0\in \partial g(x^*)+N_C(x^*)$, where $\partial g(.)$ denotes the subdifferential of $g$ and $N_C(x^*)$ is the normal cone 
of  $C$ at $x^*$.
\end{lemma}
%%%%%%%%%%%%%%%%%%%%%%%%%%%%
%%%%%%%%%%%%%%%%%%%%%%%%%%%%%%%%%%%%%
Based on Lemma $\ref{lem.Equivalent_MinPro}$, we obtain the following central lemma which is used to prove the convergence 
of Algorithm $\ref{algor1}$.
%%%%%%%%%%%%%%%%%%%%%%%%%%%%%%%%%%%%%
\begin{lemma}\label{lem1}
Assume that $x^*\in EP(f,C)$. Let $\left\{x_n\right\},\left\{y_n\right\}$ be the sequences generated by Algorithm $\ref{algor1}$. Then, there holds the relation
$$
 ||y_{n+1} - x^*||^2\leq ||x_n-x^*||^2+\epsilon_n,
$$
where $\epsilon_n$ is defined by Step 2 of Algorithm $\ref{algor1}$.
\end{lemma}
%%%%%%%%%%%%%%%%%%%%%%%%%%%
%%%%%%%%%%%%%%%%%%%%%%%%
%%%%%%%%%%%%%%%%%%%%%%%%%%%
\begin{proof}
From the definition of $y_{n+1}$ and Lemma $\ref{lem.Equivalent_MinPro}$,
\begin{equation*}
0\in \partial_2\left(\lambda f(y_n,y)+\frac{1}{2}||x_n-y||^2\right)(y_{n+1})+N_C(y_{n+1}).
\end{equation*}
Thus, there exist $w\in \partial_2 f(y_n,y_{n+1}):=\partial f(y_n,.)(y_{n+1})$ and $\bar{w}\in N_C(y_{n+1})$ such that
\begin{equation*}
\lambda w+y_{n+1}-x_n+\bar{w}=0.
\end{equation*}
Hence,
\begin{equation*}
\left\langle y_{n+1}-x_n,y-y_{n+1}\right\rangle=\lambda \left\langle w,y_{n+1}-y\right\rangle+\left\langle \bar{w},y_{n+1}-y\right\rangle,~\forall y\in C.
\end{equation*}
This together with the definition of $N_C$ implies that
\begin{equation*}
\left\langle y_{n+1}-x_n,y-y_{n+1}\right\rangle\ge\lambda \left\langle w,y_{n+1}-y\right\rangle,~\forall y\in C.
\end{equation*}
By $w\in \partial_2 f(y_n,y_{n+1})$,
\begin{equation*}
f(y_n,y)-f(y_n,y_{n+1})\ge \left\langle w,y-y_{n+1}\right\rangle,~\forall y\in C.
\end{equation*}
From the last two inequalities, we obtain
\begin{equation}\label{eq:1}
\left\langle y_{n+1}-x_n,y-y_{n+1}\right\rangle\ge \lambda\left(f(y_n,y_{n+1})-f(y_n,y)\right),~\forall y\in C.
\end{equation}
Similarly, by replacing $n+1$ by $n$, we also have
\begin{equation}\label{eq:2}
\left\langle y_{n}-x_{n-1},y-y_{n}\right\rangle\ge \lambda\left(f(y_{n-1},y_{n})-f(y_{n-1},y)\right),~\forall y\in C.
\end{equation}
Substituting $y=y_{n+1}$ onto $(\ref{eq:2})$ and a straightforward computation yield
\begin{equation}\label{eq:3}
\lambda\left(f(y_{n-1},y_{n+1})-f(y_{n-1},y_{n})\right)\ge\left\langle y_{n}-x_{n-1},y_{n}-y_{n+1}\right\rangle .
\end{equation}
Substituting $y=x^*$ onto $(\ref{eq:1})$ we also obtain
\begin{equation}\label{eq:3*}
\left\langle y_{n+1}-x_n,x^*-y_{n+1}\right\rangle\ge \lambda\left(f(y_n,y_{n+1})-f(y_n,x^*)\right).
\end{equation}
Since $x^*\in EP(f,C)$ and $y_n\in C$, $f(x^*,y_n)\ge 0$. Thus, from the pseudomonotonicity of $f$ one has $f(y_n,x^*)\le 0$. 
This together with $(\ref{eq:3*})$ implies that 
\begin{equation}\label{eq:4}
\left\langle y_{n+1}-x_n,x^*-y_{n+1}\right\rangle\ge \lambda f(y_n,y_{n+1}).
\end{equation}
By the Lipschitz-type continuity of $f$,
\begin{equation*}
f(y_{n-1},y_n)+f(y_n,y_{n+1})\ge f(y_{n-1},y_{n+1})-c_1||y_{n-1}-y_n||^2-c_2||y_n-y_{n+1}||^2.
\end{equation*}
Thus,
\begin{equation}\label{eq:5}
f(y_n,y_{n+1})\ge f(y_{n-1},y_{n+1})-f(y_{n-1},y_n)-c_1||y_{n-1}-y_n||^2-c_2||y_n-y_{n+1}||^2.
\end{equation}
The relations $(\ref{eq:4})$ and $(\ref{eq:5})$ lead to
\begin{eqnarray*}
\left\langle y_{n+1}-x_n,x^*-y_{n+1}\right\rangle&\ge& \lambda\left\{f(y_{n-1},y_{n+1})-f(y_{n-1},y_n)\right\}\nonumber\\
&&-\lambda c_1||y_{n-1}-y_n||^2-\lambda c_2||y_n-y_{n+1}||^2.\label{eq:6}
\end{eqnarray*}
This together with the relation $(\ref{eq:3})$ implies that
\begin{eqnarray*}
\left\langle y_{n+1}-x_n,x^*-y_{n+1}\right\rangle&\ge& \left\langle y_{n}-x_{n-1},y_{n}-y_{n+1}\right\rangle-\lambda c_1||y_{n-1}-y_n||^2\nonumber\\
&&-\lambda c_2||y_n-y_{n+1}||^2.
\end{eqnarray*}
Thus,
\begin{eqnarray}\label{eq:7}
2\left\langle y_{n+1}-x_n,x^*-y_{n+1}\right\rangle&-& 2\left\langle y_{n}-x_{n-1},y_{n}-y_{n+1}\right\rangle\ge-2\lambda c_1||y_{n-1}-y_n||^2\nonumber\\
&&-2\lambda c_2||y_n-y_{n+1}||^2.
\end{eqnarray}
We have the following fact
\begin{eqnarray}
&&2\left\langle y_{n+1}-x_n,x^*-y_{n+1}\right\rangle=||x_n-x^*||^2-||y_{n+1}-x^*||^2-||x_n-y_{n+1}||^2\nonumber\\ 
&=& ||x_n-x^*||^2-||y_{n+1}-x^*||^2-||x_n-x_{n-1}||^2-2\left\langle x_n-x_{n-1},x_{n-1}-y_{n+1}\right\rangle\nonumber\\
&&-||x_{n-1}-y_{n+1}||^2\nonumber\\
&=&||x_n-x^*||^2-||y_{n+1}-x^*||^2-||x_n-x_{n-1}||^2-2\left\langle x_n-x_{n-1},x_{n-1}-y_{n+1}\right\rangle\nonumber\\
&&-||x_{n-1}-y_{n}||^2-2\left\langle x_{n-1}-y_{n},y_{n}-y_{n+1}\right\rangle-||y_{n}-y_{n+1}||^2.\label{eq:8**}
\end{eqnarray}
By the triangle, Cauchy-Schwarz and Cauchy inequalities,
\begin{eqnarray*}
&-&2\left\langle x_n-x_{n-1},x_{n-1}-y_{n+1}\right\rangle\le2||x_n-x_{n-1}||||x_{n-1}-y_{n+1}||\\
&&\le2||x_n-x_{n-1}||||x_{n-1}-y_{n}||+2||x_n-x_{n-1}||||y_{n}-y_{n+1}||\\ 
&&\le|| x_n-x_{n-1}||^2+||x_{n-1}-y_{n}||^2+k|| x_n-x_{n-1}||^2+\frac{1}{k}||y_{n}-y_{n+1}||^2.
\end{eqnarray*}
This together with $(\ref{eq:8**})$ implies that 
\begin{align}
&2\left\langle y_{n+1}-x_n,x^*-y_{n+1}\right\rangle\le ||x_n-x^*||^2-||y_{n+1}-x^*||^2+k|| x_n-x_{n-1}||^2\nonumber\\
&+2\left\langle y_n-x_{n-1},y_{n}-y_{n+1}\right\rangle+\left(\frac{1}{k}-1\right)||y_{n}-y_{n+1}||^2\label{eq:8***}.
\end{align}
Thus, 
\begin{eqnarray}
&&2\left\langle y_{n+1}-x_n,x^*-y_{n+1}\right\rangle-2\left\langle y_n-x_{n-1},y_{n}-y_{n+1}\right\rangle\le||x_n-x^*||^2\nonumber\\
&&-||y_{n+1}-x^*||^2+k|| x_n-x_{n-1}||^2+\left(\frac{1}{k}-1\right)||y_{n}-y_{n+1}||^2\label{eq:8***}.
\end{eqnarray}
Combining $(\ref{eq:7})$ and $(\ref{eq:8***})$ we obtain 
\begin{eqnarray*}
-2\lambda c_1||y_{n-1}-y_n||^2&-&2\lambda c_2||y_n-y_{n+1}||^2\le||x_n-x^*||^2-||y_{n+1}-x^*||^2\\
&+&k|| x_n-x_{n-1}||^2+\left(\frac{1}{k}-1\right)||y_{n}-y_{n+1}||^2. 
\end{eqnarray*}
Thus, from the definition of $\epsilon_n$ we obtain
\begin{eqnarray*}
||y_{n+1}-x^*||^2&\leq&||x_n-x^*||^2+k|| x_n-x_{n-1}||^2+2\lambda c_1||y_{n-1}-y_n||^2\\
&&-\left(1-\frac{1}{k}-2\lambda c_2\right)||y_{n}-y_{n+1}||^2. \\
&=&||x_n-x^*||^2+\epsilon_n.
\end{eqnarray*}
Lemma $\ref{lem1}$ is proved.
\end{proof}
%%%%%%%%%%%%%%%%%%%%%%%
\begin{lemma}\label{lem4}
Let $\left\{x_n\right\},\left\{y_n\right\}$ be the sequences generated by Algorithm $\ref{algor1}$. Then, there hold the following relations
\begin{itemize}
\item [$(i)$] $EP(f,C)\subset C_n\cap Q_n$ for all $n\ge 0$.
\item [$(ii)$] $\lim_{n\to\infty}||x_{n+1}-x_n||=\lim_{n\to\infty}||y_n-x_n||=\lim_{n\to\infty}||y_{n+1} - y_{n}||=0.$
\end{itemize}
\end{lemma}
%%%%%%%%%%%%%%%%%%%%%%%%%%%
\begin{proof}
(i). From the definitions of $C_n$ and $Q_n$, we see that they are the half-spaces. Thus, $C_n$ and $Q_n$ 
are closed and convex for all $n\ge 0$. Lemma $\ref{lem1}$ and the definition of $C_n$ ensure that $EP(f,C)\subset C_n$ for all $n\ge 0$. 
It is clear that $EP(f,C)\subset C_0\cap Q_0$. Assume that $EP(f,C)\subset C_n\cap Q_n$ for some $n\ge 0$. From $x_{n+1}=P_{C_n\cap Q_n}(x_0)$ 
and Lemma $\ref{lem.PropertyPC}$(iii) we see that $\left\langle z-x_{n+1},x_0-x_{n+1}\right\rangle\le 0$ for all $z\in C_n\cap Q_n$. This 
is also true for all $z\in F$ because $EP(f,C)\subset C_n\cap Q_n$. From the definition of $Q_{n+1}$, $EP(f,C)\subset Q_{n+1}$ or 
$EP(f,C)\subset C_{n+1}\cap Q_{n+1}$. By the induction, $EP(f,C)\subset C_n\cap Q_n$ for all $n\ge 0$. Since $EP(f,C)$ is nonempty, so $C_n\cap Q_n$ is. 
Thus, $P_{C_n\cap Q_n}(x_0)$ is well-defined.\\
(ii). From the definition of $Q_n$ and Lemma $\ref{lem.PropertyPC}$(iii.), $x_n=P_{Q_n}(x_0)$. Thus, from Lemma $\ref{lem.PropertyPC}$(ii) we have
\begin{equation}\label{eq:8}
||z-x_n||^2\le||z-x_0||^2-||x_n-x_0||^2,~\forall z\in Q_n.
\end{equation}
Substituting $z=x^\dagger:=P_{EP(f,C))}(x_0)\in Q_n$ onto $(\ref{eq:8})$, one has 
\begin{equation}\label{eq:8*}
||x^\dagger-x_0||^2-||x_n-x_0||^2\ge ||x^\dagger-x_n||^2\ge 0.
\end{equation}
Thus, the sequence $\left\{||x_n-x_0||\right\}$, therefore $\left\{x_n\right\}$, are bounded. Substituting $z=x_{n+1}\in Q_n$ onto $(\ref{eq:8})$, one also has 
\begin{equation}\label{eq:9}
0\le||x_{n+1}-x_n||^2\le||x_{n+1}-x_0||^2-||x_n-x_0||^2.
\end{equation}
This implies that $\left\{||x_n-x_0||\right\}$ is non-decreasing. Hence, there exists the limit of $\left\{||x_n-x_0||\right\}$. 
By $(\ref{eq:9})$, 
$$
\sum_{n=1}^K||x_{n+1}-x_n||^2\le ||x_{K+1}-x_0||^2-||x_1-x_0||^2,~\forall K\ge 1.
$$
Passing the limit in the last inequality as $K\to\infty$, we obtain
\begin{equation}\label{eq:10*}
\sum_{n=1}^\infty||x_{n+1}-x_n||^2<+\infty.
\end{equation}
Thus,
\begin{equation}\label{eq:11*}
\lim_{n\to\infty}||x_{n+1}-x_n||=0.
\end{equation}
From the definition of $C_n$ and $x_{n+1}\in C_n$,
\begin{equation}\label{eq:12}
 ||y_{n+1} - x_{n+1}||^2\leq ||x_n-x_{n+1}||^2+\epsilon_n.
\end{equation}
Set $\alpha_n=||y_{n+1} - x_{n+1}||^2$, $\beta_n=||x_n-x_{n+1}||^2+k||x_n-x_{n-1}||^2$, $\gamma_n=||y_n-y_{n-1}||^2$, $\beta=2\lambda c_2$, 
and  $\alpha=1-\frac{1}{k}-2\lambda c_1$. From the definition of $\epsilon_n$, $\epsilon_n=k||x_n-x_{n-1}||^2+\beta \gamma_n-\alpha \gamma_{n+1}$. 
Thus, from $(\ref{eq:12})$, 
\begin{equation}\label{eq:13*}
\alpha_n\le \beta_n+\beta \gamma_n-\alpha \gamma_{n+1}.
\end{equation}
From the hypothesises of $\lambda, ~k$ and $(\ref{eq:10*})$, we see that $\alpha>\beta\ge 0$ and $\sum_{n=1}^\infty \beta_n<+\infty$. Lemma 
$\ref{lem.technique}$ and $(\ref{eq:13*})$ imply that $\alpha_n\to 0$, or 
\begin{equation}\label{eq:14}
\lim_{n\to\infty}||y_{n+1}- x_{n+1}||=0.
\end{equation}
This together with the relation $(\ref{eq:11*})$ and the inequality $||y_{n+1}-y_{n}||\le||y_{n+1}-x_{n+1}||+||x_{n+1}-x_n||+||x_{n}-y_{n}||$ implies that 
\begin{equation}\label{eq:15}
\lim_{n\to\infty}||y_{n+1} - y_{n}||=0.
\end{equation}
In addition, the sequence $\left\{y_{n}\right\}$ is also bounded because of the boundedness of $\left\{x_n\right\}$. Lemma $\ref{lem4}$ is proved.\\
\end{proof}
%%%%%%%%%%%%%%%%%%%%%%%%%%%%%%%%%
Thanks to Lemma $\ref{lem2}$, we see that if Algorithm $\ref{algor1}$ terminates at the iterate $n$ then a solution of EP can be found. 
Otherwise, if Algorithm $\ref{algor1}$ 
does not terminate then we have the following main result.
%%%%%%%%%%%%%%%%%%%%%%%%%%%%%%%%%%%%%%%%%%%%%%%%
\begin{theorem}\label{theo.1}
Let $C$ be a nonempty closed convex subset of a real Hilbert space $H$. Assume that the bifunction $f$ satisfies all conditions $\rm (A1)-(A4)$. 
In addition the solution set $EP(f,C)$ is nonempty. Then, the sequences $\left\{x_n\right\}$, $\left\{y_n\right\}$ generated by Algorithm 
$\ref{algor1}$ converge strongly to $P_{EP(f,C)}(x_0)$.
\end{theorem}
%%%%%%%%%%%%%%%%%%%%%%%%%%%%%%%
\begin{proof}
From Lemma $\ref{lem4}$, the sequence $\left\{x_n\right\}$ is bounded. Assume that $p$ is any weak cluster point of $\left\{x_n\right\}$. 
Without loss of generality, we can write $x_n\rightharpoonup p$ as $n\to \infty$. Thus, $y_n\rightharpoonup p$ because $||x_n-y_n||\to 0$. 
Now, we show that $p\in EP(f,C)$.  
From $(\ref{eq:1})$, we get 
\begin{equation}\label{eq:15}
\lambda\left(f(y_n,y)-f(y_n,y_{n+1})\right)\ge \left\langle x_n-y_{n+1},y-y_{n+1}\right\rangle,~\forall y\in C.
\end{equation}
Passing the limit in $(\ref{eq:15})$ as $n\to\infty$ and using Lemma $\ref{lem4}$(ii), the bounedness of $\left\{y_n\right\}$ 
and $\lambda>0$ we obtain $f(p,y)\ge 0$ for all $y\in C$. Thus, $p\in EP(f,C)$. From the inequality $(\ref{eq:8*})$, we get
$$ ||x_n-x_0||\le ||x^\dagger-x_0||,$$
where $x^\dagger=P_{EP(f,C)}(x_0)$. By the weak lower semicontinuity of the norm $||.||$ and $x_n\rightharpoonup p$, we have
\begin{equation*}
||p-x_0||\le \lim_{n\to\infty}\inf||x_{n}-x_0||\le \lim_{n\to\infty}\sup||x_{n}-x_0||\le||x^\dagger-x_0||.
\end{equation*}
By the definition of $x^\dagger$, $p=x^\dagger$ and $\lim_{n\to\infty}||x_{n}-x_0||=||x^\dagger-x_0||$. Thus, $\lim_{n\to\infty}||x_{n}||=||x^\dagger||$. 
By the Kadec-Klee property of the Hilbert space $H$, we have $x_{n}\to x^\dagger=P_{EP(f,C)}x_0$ as $n\to\infty$. From Lemma $\ref{lem4}$, we 
also see that $\left\{y_n\right\}$ converges strongly to $P_{EP(f,C)}x_0$. Theorem $\ref{theo.1}$ is proved.
\end{proof}
%%%%%%%%%%%%%%%%%%%%%%%%%%%%%%%%%%
\section{Applications}\label{Appl}
%%%%%%%%%%%%%%%%%%%%%%%%%%%%
In this section, we introduce several applications of Algorithm \ref{algor1} to Gato differentiable EPs and multivalued variational inequalities.
%%%%%%%%%%%%%%%%%%%%%%%%%%%%%%%
\subsection{Gato differentiable equilibrium problems}
%%%%%%%%%%%%%%%%%%%%%%%%%%%%%%%%%%%%%
We consider EPs for Gato differentiable bifunctions. We denote $\nabla_2 f(x,y)$ by the Gato derivative of 
the function $f(x,.)$ at $y$. For solving EP $(\ref{eq:EP})$, we assume that the bifunction $f$ satisfies the following conditions:
\begin{itemize}
\item[\rm (B1).] $f$ is monotone on $C$ and $f(x,x)=0$ for all $x\in C$;
\item [\rm (B2).] $f(x,.)$ is convex and Gato differentiable on $C$;
\item [\rm (B3).] There exists a constant $L>0$ such that 
$$||\nabla_2 f (x,x)-\nabla_2 f (y,y)||\le L||x-y||~\forall x,y\in C;$$
\item [\rm (B4).] $\lim\limits_{t\to 0^+}\sup f(x+t(z-x),y)\le f(x,y)$ for all $x,y\in C$.
\end{itemize}  
%%%%%%%%%%%%%%%%%%%%%%%%%%%%%%%%%
\begin{remark}
If EP $(\ref{eq:EP})$ is reduced to VIP $(\ref{VIP})$ for the operator $A:C\to H$ then the condition $\rm (B3)$ is equivalent to the 
Lipschitzianity of $A$ with the constant $L>0$.
\end{remark}
%%%%%%%%%%%%%%%%%%%%%%%%%%%%%%%
%%%%%%%%%%%%%%%%%%%%%%%%%%%%%
We need the following results.
\begin{lemma} \cite[Lemma 2]{LSV2011}\label{lem3}
Suppose that the conditions $\rm (B1), (B2), (B4)$ hold. Then, 
\begin{itemize}
\item [$\rm i.$] The operator $A(x)=\nabla_2 f(x,x)$ is monotone on $C$.
\item [$\rm ii.$] $EP(f,C)=VI(A,C)$.
\end{itemize}
\end{lemma}
%%%%%%%%%%%%%%%%%%%%%%%%%%%%%%%%%%%%
\begin{lemma}\label{lem4}
Assume that $A:C\to H$ is a L - Lipschitz continuous operator and the bifunction $f:C\times C\to \Re$ is defined by $f(x,y)=\left\langle A(x),y-x\right\rangle$ 
for all $x,y\in C$. Then
\begin{itemize}
\item [$\rm i.$] $f$ is Lipschitz-type continuous on $C$ with $c_1=c_2=L/2$.
\item [$\rm ii.$] $z=\underset{t\in C}{\rm argmin} \{ \lambda f(y,t)+\frac{1}{2}||x-t||^2\}$ if and only if $z=P_C(x-\lambda A(y))$, where $\lambda>0$.
\end{itemize}
\end{lemma}
%%%%%%%%%%%%%%%%%%%%%%%%%%%%%%%%%%
\begin{proof}
i. From the $L$ - Lipschitz continuity of $A$, the Cauchy-Schwarz and Cauchy inequalities, we have
\begin{eqnarray*}
f(x,y)&+&f(y,z)-f(x,z)=\left\langle A(x)-A(y),y-z\right\rangle\ge-||A(x)-A(y)||||y-z||\\
&\ge&-L||x-y||||y-z||\ge-\frac{L}{2}||x-y||^2-\frac{L}{2}||y-z||^2.
\end{eqnarray*}
This implies that $f$ is Lipschitz-type continuous on $C$ with $c_1=c_2=L/2$.\\
ii. From the definition of $f$, we have
 \begin{eqnarray*}
z&=&\underset{t\in C}{\rm argmin} \{ \lambda \left\langle A(y),t-y\right\rangle+\frac{1}{2}||x-t||^2\}\\ 
&=&\underset{t\in C}{\rm argmin} \{\frac{1}{2}||t-(x-\lambda A(y))||^2-\frac{\lambda^2}{2}||A(y)||^2-\lambda \left\langle A(y),y-x\right\rangle\}\\
&=&\underset{t\in C}{\rm argmin} \{\frac{1}{2}||t-(x-\lambda A(y))||^2\}\\
&=&P_C\left(x-\lambda A(y)\right)
\end{eqnarray*}
in which the third equality is followed from the fact that $\underset{t\in C}{\rm argmin}\left\{g(t)+a\right\}=\underset{t\in C}{\rm argmin}\left\{g(t)\right\}$ 
and the last equality is true because of the definition of the metric projection. Lemma \ref{lem4} is proved.
\end{proof}
Thanks to Lemma \ref{lem3}, instead of EP (\ref{eq:EP}) we solve VIP (\ref{VIP}) for the operator $A(x)=\nabla_2 f(x,x)$ onto $C$. 
It is emphasized that $\rm (B2)$ and $\rm (B3)$ are slightly strong conditions. However, in this case, we can use the existing methods for VIPs to solve 
EPs. For instance, using the subgradient extragradient method \cite[Algorithm 3.6]{CGR2011} we obtain the following hybrid algorithm for solving 
EP (\ref{eq:EP})
\begin{equation}\label{CGR2011}
\begin{cases}
y_n=P_{C}(x_n-\lambda \nabla_2 f(x_n,x_n)),\\
z_n=\alpha_n x_n+(1-\alpha_n)P_{T_n}(x_n-\lambda \nabla_2 f(y_n,y_n)),\\
C_n=\left\{z\in H:||z_{n}-z||^2\leq ||x_n-z||^2\right\},\\
Q_n=\left\{z\in H: \left\langle x_0-x_n,z-x_n\right\rangle\le 0\right\},\\
x_{n+1}=P_{C_n\cap Q_n}(x_0),
\end{cases}
\end{equation}
where $T_n=\left\{z\in H:\left\langle x_n-\lambda \nabla_2 f(x_n,x_n)-y_n,z-y_n\right\rangle\le 0\right\}$. If the conditions $\rm (B1)-(B4)$ 
hold for all $x,y \in H$ then $\left\{x_n\right\}$ generated by (\ref{CGR2011}) converges 
strongly to $P_{EP(f,C)}(x_0)$.

In this subsection, we introduce the following strong convergence result.
%%%%%%%%%%%%%%%%%%%%%%%%%%%%%
%%%%%%%%%%%%%%%%%%%%%%%%%%%%%%%%%%
\begin{theorem}\label{theo3}
Let $C$ be a nonempty closed convex subset of a real Hilbert space $H$. Assume that the bifunction $f$ satisfies all conditions $\rm (B1)-(B4)$ 
such that $EP(f,C)$ is nonempty. Let $\left\{x_n\right\}$ be the sequence generated by the following manner: 
$x_0=x_1 \in H, ~y_0=y_1\in C,~C_0=Q_0=H$ and
\begin{equation}\label{H2015a}
\left \{
\begin{array}{ll}
y_{n+1}=  P_C\left(x_n-\lambda\nabla_2f(y_n,y_n)\right),\\
C_n=\left\{z\in H:||y_{n+1}-z||^2\leq ||x_n-z||^2+\epsilon_n \right\},\\
Q_n=\left\{z\in H: \left\langle x_0-x_n,z-x_n\right\rangle\le 0\right\},\\
x_{n+1}=P_{C_n\cap Q_n}(x_0),
\end{array}
\right.
\end{equation}
where $\epsilon_n, \lambda, k$ are defined as in Algorithm \ref{algor1} with $c_1=c_2=L/2$.  
Then, the sequence $\left\{x_n\right\}$ converges strongly to $P_{EP(f,C)}(x_0)$.
\end{theorem}
%%%%%%%%%%%%%%%%%%%%%%%%%%%%%%%%%
\begin{proof}
Set $F(x,y)=\left\langle A(x),y-x\right\rangle$ for all $x,y\in C$, where $A(x)=\nabla_2 f(x,x)$. Lemma $\ref{lem3}$.ii. ensures that $EP(F,C)=EP(f,C)$. 
The bifunction $F$ satisfies the conditions 
$\rm (A3)$ and $\rm (A4)$ automatically. From Lemma \ref{lem3}.i., we see that $F$ is monotone, and so it is also pseudomonotone 
or $F$ satisfies the condition $\rm (A1)$. Lemma \ref{lem4}.i. and $\rm (B3)$ ensure that the condition $\rm (A2)$ holds for the bifunction $F$. 
From Step 1 of Algorithm $\ref{algor1}$ and Lemma \ref{lem4}.ii., $y_{n+1}=  P_C\left(x_n-\lambda\nabla_2f(y_n,y_n)\right)$. Thus, 
Theorem \ref{theo3} is directly followed from Theorem \ref{theo.1} for $f=F$.
\end{proof}
%%%%%%%%%%%%%%%%%%%%%%%%%%%%%%%%%%%%%%%%%%
\subsection{Multivalued variational inequalities}
%%%%%%%%%%%%%%%%%%%%%%%%%%%%%%%%%%%%%
In this subsection, we consider the following multivalued variational inequality problem (MVIP)
\begin{equation}\label{MVIP}
\begin{cases}
\mbox{Find}~x^*\in C ~\mbox{and}~v^*\in A(x^*)~\mbox{such that}\\
\left\langle v^*,y-x^*\right\rangle\ge 0,\forall y\in C, 
\end{cases}
\end{equation}
where $A:C\to 2^H$ is a multivalued compact operator. For a pair $x,y\in C$, we put 
\begin{equation}\label{eq:f}
f(x,y)=\underset{u\in A(x)}{\sup}\left\langle u,y-x\right\rangle.
\end{equation}
It is easy to show that $x^*$ is a solution of MVIP (\ref{MVIP}) if and only if $x^*$ is a solution of EP for the bifunction $f$ on $C$. We recall the following 
definitions.
\begin{definition}
A multivalued operator $A:C\to 2^H$ is said to be: 
\begin{itemize}
\item [$\rm i.$] monotone on $C$ if 
$$\left\langle u-v,x-y\right\rangle\ge 0~\forall x,y\in C,~\forall u\in A(x),~\forall v\in A(y); $$
\item [$\rm ii.$] pseudomonotone on $C$ if 
$$\left\langle u,x-y\right\rangle\ge 0\Longrightarrow \left\langle v,y-x\right\rangle\le 0 ~\forall x,y\in C,~\forall u\in A(x),~\forall v\in A(y); $$
\item [$\rm iii.$] $L$ - Lipschitz continuous if there exists a positive constant $L$ such that 
$$\underset{u\in A(x)}{\sup}\underset{v\in A(y)}{\inf}||u-v||\le L||x-y||~\forall x,y\in C.$$
\end{itemize}
\end{definition}
\begin{remark}
If we denote $h(C_1,C_2)$ by the Hausdorff distance between two sets $C_1$ and $C_2$ then the definition iii. means that
$$h(A(x),A(y))\le L||x-y||~\forall x,y\in C.$$
\end{remark}
We can easily check that if $A$ is pseudomonotone and $L$ - Lipschitz continuous then $f$ is also pseudomonotone and Lipschitz-type 
continuous with two constants $c_1=c_2=L/2$. Note that, when $A$ is singlevalued then Algorithm $\ref{algor1}$ becomes the \textit{hybrid 
algorithm without the extrapolation step} for variational inequalities \cite{MS2015}. When $A$ is multivalued then Algorithm $\ref{algor1}$ 
can be applied for the bifunction $f$ defined by (\ref{eq:f}). A disadvantage of performing Algorithm $\ref{algor1}$ in this case is that it is not 
easy to chose an approximation of the bifunction $f(x,y)$. In fact, we can prove the strong convergence of the following algorithm
\begin{equation}\label{H2015a}
\left \{
\begin{array}{ll}
x_0=x_1 \in H, ~y_0=y_1\in C,\\
y_{n+1}=  P_C\left(x_n-\lambda u_n\right),~u_n\in A(y_n),\\
C_n=\left\{z\in H:||y_{n+1}-z||^2\leq ||x_n-z||^2+\epsilon_n \right\},\\
Q_n=\left\{z\in H: \left\langle x_0-x_n,z-x_n\right\rangle\le 0\right\},\\
x_{n+1}=P_{C_n\cap Q_n}(x_0),
\end{array}
\right.
\end{equation}
where $\epsilon_n, \lambda, k$ are defined as in Algorithm \ref{algor1} with $c_1=c_2=L/2$.  
\section{Numerical examples}\label{Numer}
%%%%%%%%%%%%%%%%%%%%%%%%%%%%%%%%%
In this section, we consider two previously known academic numberical examples in Euclidean spaces. The purpose of these experiments is to 
illustrate the convergence of Algorithm $\ref{algor1}$ and compare its efficiency with Algorithms $\ref{VSH2013a}$, $\ref{VSH2013b}$ 
and $\ref{DHM2014}$. Of course, there are many mathematical models for EPs in infinite dimensional Hilbert spaces, see, for instance 
\cite{BO1994} and the norm convergence of algorithms is more necessary than the weak convergence. The ability of the implementation of these 
algorithms has been discussed in Sections $\ref{intro}$ and \ref{algor}. Note that Algorithm $\ref{VSH2013b}$, in general, is difficult to compute numerical 
experiments because of the complexity of the sets $C_n$. However, in the examples below, the feasible set $C$ is a polyhedron expressed by 
$Ax\le b$, where $A$ is a matrix, $b$ is a vector. Thus, from the definition of $C_n$ in Algorithm $\ref{VSH2013b}$, we see that it
is also a polyhedron and $C_0=\left\{x:A_0x\le b_0\right\}$ with $A_0=A,~b_0=b$. After that $C_{n+1}=\left\{x:A_{n+1}x\le b_{n+1}\right\}$ 
can be sequentially constructed by adding a linear inequality constraint to the set of constraints of $C_n$. This is performed in MATLAB version 
7.0 and the number of constraints increases when $n$ increases. The sets $C_n, Q_n$ in Algorithms $\ref{VSH2013a}$ and $\ref{DHM2014}$ 
are simply constructed more at each step. While the sets $C_n,~Q_n$ in Algorithm $\ref{algor1}$ are two half-spaces, so we use the explicit formula 
in \cite{CH2005,SS2000} to compute $x_{n+1}$. All convex quadratic optimization programs and the projections on polyhedrons can solved easily 
by the MALAB Optimization Toolbox where the projections are equivalently rewriten to the distance optimization programs. The algorithms are performed on a 
PC Desktop Intel(R) Core(TM) i5-3210M CPU @ 2.50GHz 2.50 GHz, RAM 2.00 GB. For a given tolerance $TOL$, we compare numbers of iterates (Iter.) 
and execution time (CPU in sec.) of mentioned algorithms above with chosing different starting points.

\textit{Example 1.} We consider the bifunction $f:C\times C\to \Re$ in $\Re^{2}$ proposed in \cite[Example 3]{MPPP2012} as 
$f(x,y)=(x_1+x_2-1)(y_1-x_1)+(x_1+x_2-1)(y_2-x_2)$ and the feasible set $C=[0,1]\times[0,1]$. It is easy to show that $f$ 
is monotone (so pseudomonotone) and Lipschitz-type continuous with $c_1=c_2=1$. The solution set of EP for $f$ on $C$ is 
$EP(f,C)=\left\{x\in C:x_1+x_2-1=0\right\}$. 
In this example, for a starting point $x_0$ then the sequence $\left\{x_n\right\}$ generated by Algorithms 
$\ref{algor1}$, $\ref{VSH2013a}$, $\ref{VSH2013b}$ and $\ref{DHM2014}$ converges strongly to 
$x^\dagger:=P_{EP(f,C)}(x_0)$ which is easily known because $EP(f,C)$ is explicit. The termination criterion in all algorithms 
is $||x_n-x^\dagger||\le TOL=0.001$.
The parameters are chosen as follows $\lambda=0.2,~k=6,~\eta=0.5$. In Algorithm $\ref{algor1}$, we chose 
$x_1=x_0,~y_0=y_1=(0,0)^T$. The results are shown in Table $\ref{tab:1}$.

%%%%%%%%%%%%%%%%%%%%%%%%%%%%%%%%
\begin{table}[ht]\caption{Results for given starting points in \textit{Example 1}.}\label{tab:1}
\medskip\begin{center}
\begin{tabular}{|c|c|c|c|c|c|c|c|c|}
\hline
 $x_0$& \multicolumn{4}{c|}{Iter.} &\multicolumn{4}{c|}{CPU in sec.}
\\ \cline{2-9}
  & Alg. $\ref{algor1}$ &Alg. $\ref{VSH2013a}$&Alg. $\ref{VSH2013b}$&Alg. $\ref{DHM2014}$& Alg. $\ref{algor1}$ &Alg. $\ref{VSH2013a}$&Alg. $\ref{VSH2013b}$&Alg. $\ref{DHM2014}$ \\ \hline
(2,5)&57&123&57&68&1.57 &2.58&1.59&1.98 \\ \hline
(5,5)&51&95&55&65& 1.50&1.58&1.54&1.46 \\ \hline
(4,4.5)&54&97&64&65& 1.58&1.98&1.80&2.12 \\ \hline
(-0.75,0)&54&98&59&65& 1.57&2.04&1.75&1.90 \\ \hline
 \end{tabular}
\end{center}
\end{table}
%%%%%%%%%%%%%%%%%%%%%%%%%%%%%%%%%%%%%
%%%%%%%%%%%%%%%%%%%%%%%%%%%%%%%%

\textit{Example 2.} We consider the pseudomonotone bifunction $f$ which comes from the Nash-Cournot equilibrium model in \cite{QMH2008,SVH2011}. 
It is defined by 
$$ f(x,y)=\left\langle Px+Qy+q,y-x\right\rangle, $$
where $q\in \Re^5$, $P,~Q\in \Re^{5\times 5}$ are two matrices of order 5 such that $Q$ is symmetric, positive semidefinite and $Q-P$ is 
negative semidefinite. The feasible set $C$ is a polyhedral convex set defined by 
$$ C=\left\{x\in \Re^5:\sum_{i=1}^5 x_i\ge -1,~-5\le x_i\le 5,~i=1,\ldots,5\right\}. $$
This example is tested with $q=(1,-2,-1,2,-1)^T$,
$$
    P =
    \left(\begin{array}{ccccc}
     3.1&2&0&0&0 \\
     2&3.6&0&0&0\\
0&0&3.5&2&0\\
0&0&2&3.5&0\\
0&0&0&0&3
    \end{array}\right),~
Q =
    \left(\begin{array}{ccccc}
     1.6&1&0&0&0 \\
     1&1.6&0&0&0\\
0&0&1.5&1&0\\
0&0&1&1.5&0\\
0&0&0&0&2\\
    \end{array}\right).
$$
Two Lipschitz-type constants are $c_1=c_2=||P-Q||/2$. The parameters are $\lambda=\frac{1}{5c_1},~k=6,~\eta=0.5$. 
We also chose $x_1=x_0,~$ and $y_1=y_0=(0,0,0,0,0)^T$ in Algorithm $\ref{algor1}$. In this example, the exact solution 
is not known. Thus, the stopping criterion is used in the algorithms as $||y_n-x_n||\le TOL=0.001$. Table $\ref{tab:2}$ shows 
the results for chosing different starting points as $ x_0^1=(1,3,1,1,2)^T$, $ x_0^2=(-3,4,1,-5,6)^T$, $ x_0^3=(3,-2,1,9,-8)^T $, 
$ x_0^4=(-2,3,-1,8,8)^T $.

%%%%%%%%%%%%%%%%%%%%%%%%%%%%%%%%
\begin{table}[ht]\caption{Results for given starting points in \textit{Example 2}.}\label{tab:2}
\medskip\begin{center}
\begin{tabular}{|c|c|c|c|c|c|c|c|c|}
\hline
 $x_0$&\multicolumn{4}{c|}{Iter.} &\multicolumn{4}{c|}{CPU in sec.}
\\ \cline{2-9}
  &Alg. $\ref{algor1}$ &Alg. $\ref{VSH2013a}$&Alg. $\ref{VSH2013b}$&Alg. $\ref{DHM2014}$& Alg. $\ref{algor1}$ &Alg. $\ref{VSH2013a}$&Alg. $\ref{VSH2013b}$&Alg. $\ref{DHM2014}$ \\ \hline
$x_0^1$&225&1022&958&798 &7.10&24.33&23.53&47.36 \\ \hline
$x_0^2$&369&1666&1021&1187&11.25 &34.92&33.01&44.71 \\ \hline
$x_0^3$&435&2854&923&1226&15.22 &66.53&39.99&42.13 \\ \hline
$x_0^4$&411&2385&1501&1228&12.01&37.55&41.54 &64.01 \\ \hline
 \end{tabular}
\end{center}
\end{table}
%%%%%%%%%%%%%%%%%%%%%%%%%%%%%%%%%%%%%
Although, the study of the numerical examples here is preliminary and it is clear that EP depends on the structure of the feasible set $C$ and the bifunction $f$. 
However, the results in Tables $\ref{tab:1}$ and $\ref{tab:2}$ show the convergence of our proposed algorithm and compare its efficiency with the others.
%%%%%%%%%%%%%%%%%%%%%%%%%%%%%
\section{Concluding remarks}
The paper proposes a novel algorithm for solving EPs for a class of pseudomonotone and Lipschitz-type continuous bifunctions. By constructing the specially cutting halfspaces, 
we have designed the algorithm without the extra-steps. This is the reason which explains why our algorithm can be considered as an improvement of  some previously known algorithms. 
The strong convergence of the algorithm is proved and its efficiency is illustrated by some numerical experiments.
It is also emphasized that we still have to solve exactly an optimization problem in each step. This, in general, is a disadvantage of the algorithm 
(also, of the extragradient methods and the Armijo linesearch methods) when 
equilibrium bifunctions and feasible sets have complex structures. However, contrary to several previous algorithms, our algorithm does not only avoid 
using the extra-steps which, in general, are inherently costly but also is numerically easer at its last step because the projection is only performed onto the 
intersection of two half-spaces. The paper also help us in the design and analysis of more practical algorithms to be seen. Finally, it seems to be that the 
algorithm also has competitive advantage.
%spell_to this point %********** End of text entry *****************
%\cite{CoddLev55},\cite{2001MMA06n2},\cite{2002DE38n7}
%\bibliography{x}
%%%%%%%%%%%%%%%%%%%%%%%%%%%%%

\end{document}